\documentclass[12pt]{article}
\usepackage{amsmath}
\usepackage{amssymb,float,dsfont,authblk}
\usepackage{amsthm,graphicx,pifont}
\usepackage[hmargin=25mm,a4paper]{geometry}
\usepackage[colorlinks]{hyperref}
\newtheorem{theorem}{Theorem}[section]
\newtheorem{definition}{Definition}[section]
\newtheorem{lemma}[theorem]{Lemma}
\newtheorem{corollary}[theorem]{Corollary}
\newtheorem{conjecture}{Conjecture}[section]
\newtheorem{claim}{Claim}
\newtheorem{prop}{Proposition}[section]

\title{Properly colored even cycles in edge-colored complete balanced bipartite graphs}
\author[a]{Shanshan Guo}
\author[a]{Fei Huang\thanks{Corresponding author: Fei Huang. Email: hf@zzu.edu.cn}}
\author[a]{Jinjiang Yuan}
\author[b]{C.T. Ng}
\author[b]{T.C.E. Cheng}
\affil[a]{School of Mathematics and Statistics, Zhengzhou University, Zhengzhou, Henan, People's Republic of China \authorcr \it{ Email: 15738385820@163.com; hf@zzu.edu.cn; yuanjj@zzu.edu.cn.}}
\affil[b]{Logistics Research Centre, Department of Logistics and Maritime Studies, The Hong Kong Polytechnic University, Hong Kong SAR, People's Republic of China
\authorcr \it {Email: daniel.ng@polyu.edu.hk; edwin.cheng@polyu.edu.hk.}}

\newcommand\figcaption{\def\@captype{figure}\caption}
\newcommand\tabcaption{\def\@captype{table}\caption}
\date{}\makeatother
\begin{document}
\maketitle
\begin{abstract}
Consider a complete balanced bipartite graph $K_{n,n}$ and let $K^c_{n,n}$ be an edge-colored version of $K_{n,n}$ that is obtained from $K_{n,n}$ by having each edge assigned a certain color. A subgraph $H$ of $K^c_{n,n}$ is called properly colored (PC)
if every two adjacent edges of $H$ have distinct colors. $K_{n,n}^c$ is called properly vertex-even-pancyclic
if for every vertex $u\in V(K_{n,n}^c)$ and for every even integer $k$ with $4 \leq k \leq 2n$,
there exists a PC $k$-cycle containing $u$. The minimum color degree $\delta^c(K^c_{n,n})$ of $K^c_{n,n}$ is the largest
integer $k$ such that for every vertex $v$, there are at least $k$ distinct colors on the edges incident to $v$. In this paper we study the existence of PC even cycles in $K_{n,n}^c$. We first show that, for every integer $t\geq 3$, every $K^c_{n,n}$ with $\delta^c(K^c_{n,n})\geq \frac{2n}{3}+t$ contains a PC 2-factor $H$ such that every cycle of $H$ has a length of at least $t$. By using the probabilistic method and absorbing technique, we use the above result to further show that, for every $\varepsilon>0$, there exists an integer $n_0(\varepsilon)$ such that every $K^c_{n,n}$ with $n\geq n_0(\varepsilon)$ is properly vertex-even-pancyclic, provided that $\delta^c(K^c_{n,n})\geq (\frac{2}{3}+\varepsilon)n$.\\


\end{abstract}

\section{Introduction}

Consider a simple and finite graph $G$. We use the following notation and terminology in this paper.

\noindent$\bullet$ $V(G)$ and $E(G)$ are the sets of vertices and edges of $G$, respectively.

\noindent$\bullet$  $|G|=|V(G)|$ is the number of vertices of $G$.

\noindent$\bullet$ $H\subseteq G$ means that $H$ is a subgraph of $G$, i.e.,
$V(H)\subseteq V(G)$ and $E(H)\subseteq E(G)$.

\noindent$\bullet$ $G\setminus U$ with $U\subseteq V(G)$ is the subgraph obtained from $G$ by deleting the vertices of $U$ and the edges incident to $U$.

\noindent$\bullet$ $G\setminus H=G\setminus V(H)$ for a proper subgraph $H\subset G$.

\noindent$\bullet$ For two vertices $x,y\in V(G)$, the length of a shortest $xy$-path is called the distance between $x$ and $y$ in $G$ and denoted by $d_G(x, y)$.

\noindent$\bullet$  An edge-colored version of $G$, denoted by $G^c$, is obtained from $G$ by assigning to each edge of $G$ a color.
We also call $G^c$ an \emph{edge-colored graph} and  $c$ an \emph{edge coloring} of $G$.

\noindent$\bullet$ For each edge $e\in E(G)$, we use $c(e)$ to denote the color of edge $e$ in $G^c$.

\noindent$\bullet$
$G^c$ is called \emph{properly colored} (PC) if every two adjacent edges of $G$ receive distinct colors in $G^c$.

\noindent$\bullet$ $d^c_G(v)$ with $v\in V(G^c)$,
called the \emph{color degree} of $v$ in $G^c$, is the number of colors of the edges incident to $v$ in $G^c$.
When no confusion can occur, we shortly write
$d^c(v)$ for $d^c_G(v)$.

\noindent$\bullet$
$\delta^c(G)=\min\{d^c(v):v\in V(G^c)\}$
is called the \emph{minimum color degree} of $G^c$.

\noindent$\bullet$
$\Delta_{mon}(v)$ with $v\in V(G^c)$  denotes the maximum number of edges of the same color incident to $v$ in $G^c$.

\noindent$\bullet$
$\Delta_{mon}(G^c)$ is the maximum value of $\Delta_{mon}(v)$ over all the vertices $v\in V(G^c)$.

\noindent$\bullet$ For two vertices $x,y\in V(G)$,
we use $d_G(x,y)$ to denote the length of the shortest $xy$-path in $G$.

\noindent$\bullet$
A \emph{1-path-cycle} of $G^c$ is a vertex-disjoint union of exactly one PC path
of $G^c$ and a number of PC cycles of $G^c$. Note that every PC path of $G^c$ is a 1-path-cycle.

\noindent$\bullet$
A 1-path-cycle $H$ of $G^c$ is called a \emph{$1^{(k)}$-path-cycle}
if the length of every cycle of $H$ is at least $k$.

\noindent$\bullet$
A \emph{2-factor} $H$ of $G^c$ is a spanning 2-regular subgraph of $G^c$.

\noindent$\bullet$
A 2-factor $H$ of $G^c$ is called \emph{$2^{(k)}$-factor}
if the length of every cycle of $H$ is at least $k$.

\noindent$\bullet$
An edge-colored complete bipartite graph $K_{n,n}^c$ is called \emph{vertex-even-pancyclic}
if, for every vertex $u\in V(K_{n,n}^c)$ and
for every even integer $k$ with $4 \leq k \leq 2n$,
there exists a PC $k$-cycle containing vertex $u$ in $K_{n,n}^c$.

For more graph-theoretical terminology and notation not defined here, we refer the reader to \cite{BM}.

In the past few decades,
PC subgraphs in edge-colored graphs have been extensively studied in the literature.
In 1976, Daykin \cite{D} asked whether there exists a constant $\mu$ such that every edge-colored complete graph $K_n^c$ with $\Delta^{mon}(K_n^c)\le \mu n$ and $n\ge 3$ contains a PC Hamilton cycle. This question was answered independently by Bollob\'as and Erd\"os \cite{BE} with $\mu=1/69$, and Chen and Daykin \cite{CD} with $\mu=1/17$.
Bollob\'as and Erd\H{o}s  \cite{BE} further conjectured that if $\Delta_{mon}(K^c_
n)< \lfloor \frac{n}{2}\rceil$, then $K_n^c$ contains a PC Hamiltonian
cycle.  Cheng et al. \cite{CK} showed that every edge-colored graph $G^c$ with $\delta^c(G^c)\geq |G|/2$ has a PC spanning tree.
By Lo \cite{A}, for any $\varepsilon>0$, there exists an integer $N_0 = N_0(\varepsilon)$ such that every $K^c_{n}$  with $n\geq N_0$ and $\Delta_{mon}(K^c_{n})\leq (1/2 -\varepsilon)n $ contains a PC Hamiltonian cycle.
Barr \cite{B1998} proved that $K_n^c$ has a PC Hamilton path, provided that $K_n^c$ contains no monochromatic triangle.
Feng et al. \cite{FGGG} showed that $K_n^c$ contains a PC Hamilton path if and only if $K_n^c$ contains a spanning PC 1-path-cycle.

For the existence of PC subgraphs in edge-colored bipartite graphs, Bang-Jensen et al. \cite{BG} showed that a 2-edge-colored complete bipartite multigraph contains a PC Hamilton cycle if and only if it is color-connected and has a PC 2-factor.
Chen and Daykin \cite{CD} showed that $K^c_{n,n}$ contains a PC Hamiltonian cycle, provided $\Delta_{mon}(K_{n,n}^c)\le \frac{1}{25}n$.
By Alon and Gutin \cite{AG}, for any $\varepsilon>0$, there exists an integer $N_0 = N_0(\varepsilon)$ such that every $K^c_{n,n}$  with $n\geq N_0$  and $\Delta_{mon}(K_{n,n}^c)\le (1-\frac{1}{\sqrt{2}}-\varepsilon)n$ contains a PC cycle of length $l$ for all even $l$ with $4\le l \le 2n$.
Kano and Tsugaki \cite{KT} showed that
every edge-colored connected bipartite graph $G^c$ with $\delta^c(G^c)\geq |G|/3+1$
has a PC spanning tree.
In 2017, Fujita et al. \cite{FLZ} presented the following conjecture.

\begin{conjecture}[\cite{FLZ}]\label{conj}
If $\delta^c(K_{m,n})\geq \frac{m+n}{4}+1$, then each vertex of $K_{m,n}$ is contained in a PC $l$-cycle, where $l$ is any even integer with $4\leq l\leq \min\{2m,2n\}$.
\end{conjecture}

The authors in \cite{FLZ} also showed the following result to support Conjecture \ref{conj}.

\begin{theorem}[\cite{FLZ}]\label{h1}
If $\delta^c(K_{n,n})\geq n/2+1$, then every vertex of $K^c_{n,n}$ is contained in a PC 4-cycle.
\end{theorem}

Recently, about the existence of PC $2$-factor in $K^c_{n,n}$, Guo et al. \cite{GHY} presented the following result.

\begin{theorem}[\cite{GHY}]\label{t3}
$K^c_{n,n}$ has a PC 2-factor, provided that $\delta^c(K^c_{n,n})> \frac{3n}{4}$.
\end{theorem}

For more results on PC cycles and paths in edge-colored graphs, we recommend \cite{BG1, BGY, BE, COY, CHY, CL, CH, G, H, LNZ, LW, A2, A1, X}, and Chapter 16 of \cite{BG}.

In this paper we study the existence of PC $2$-factors in $K^c_{n,n}$ and the vertex-even-pancyclicity of $K^c_{n,n}$.
Our main results are the following Theorems \ref{t1} and \ref{t2}.

\begin{theorem}\label{t1}
Let $t\geq 3$ be an integer.
Every $K^c_{n,n}$ with $\delta^c(K^c_{n,n})\geq \frac{2n}{3}+t$ contains a PC 2-factor $H$ such that every cycle of $H$ has a length of at least $t$.
\end{theorem}

\begin{theorem}\label{t2}
For every $\varepsilon>0$, there exists an integer $n_0(\varepsilon)$ such that every $K^c_{n,n}$ with $n\geq n_0(\varepsilon)$
and $\delta^c(K^c_{n,n})\geq (\frac{2}{3}+\varepsilon)n$
is vertex-even-pancyclic.
\end{theorem}

From Theorem \ref{t1}, we have the following corollary that improves the result in Theorem \ref{t3} since the condition ``$\delta^c(K^c_{n,n})> \frac{3n}{4}$" is weakened to ``$\delta^c(K^c_{n,n})\geq \frac{2n}{3}+3$".

\begin{corollary}\label{c2}
Every $K^c_{n,n}$ with $\delta^c(K^c_{n,n})\geq \frac{2n}{3}+3$ contains a PC 2-factor.
\end{corollary}

Regarding the methodology to find properly a colored $2$-factor, we deploy the proof technique that we developed in \cite{GHY}
with significant modifications. A common technique is to gradually expand a desired subgraph to a larger one.
In \cite{GHY}, the authors always add two new vertices to the current desired subgraph that is the union of some vertex-disjoint PC cycles, and then form a new larger desired subgraph. In contrast, in this paper, we mainly add a new vertex to the current desired subgraph that is a 1-path-cycle to obtain a larger one, and finally we convert a spanning 1-path-cycle into a PC 2-factor.
Thus, the desired subgraphs in the two papers are different. In addition, the authors in \cite{GHY} mainly use the value
$\Delta_{mon}(K_{n,n}^c)$ in their analysis, while we focus more on the value $\delta^c(K^c_{n,n})$ in our analysis in this paper.
After establishing Theorem \ref{t1}, we use it together with the probabilistic method to derive Theorem \ref{t2}.

\section{Proof of Theorem \ref{t1}}
Let $t\geq 3$ be an integer.
Let $G=K^c_{n,n}$ be an edge-colored complete balanced bipartite graph with bipartition $(X,Y)$ and $\delta^c(G)\geq \frac{2n}{3}+t$.
This means that $n\geq 3t$. Suppose to the contrary that $G$   contains no PC $2^{(t)}$-factor.
Let $H$ be a $1^{(t)}$-path-cycle  of $G$
and let $P$ be the unique path (as a component) of $H$.
For our purpose, we may assume that the pair
$(H,P)$ is chosen such that: (i) $|H|$ is as large as possible, and (ii) subject to (i), $|P|$ is as large as possible.
We first show that $H$ is a spanning $1^{(t)}$-path-cycle  of $G$.

Note that $P$ is a PC path of $G$ and $H\setminus P$ (if $P\neq H$) consists of some
PC even cycles of $G$, each of which has a length of at least $t$.
For the case where $|P|=1$, from the fact that
$G=K_{n,n}^c$, we have $|H|< |G|$ and there must be an edge $xy\in E(G)$ such that
$x,y\in V(G)\setminus V(H\setminus P)$.
Then $H'=xy \cup (H\setminus P)$ is a new $1^{(t)}$-path-cycle  of $G$
such that $|H'|>|H|$. This contradicts the choice of $H$.
Hence, we must have $|P|\geq 2$.

Suppose that $P:=u_1u_2 \ldots u_k$.
Let $S'_1:=\{u\in N_G(u_1):c(u_1u)\neq c(u_1u_2)\}$
and $S'_k:=\{u\in N_G(u_k):c(u_ku)\neq c(u_{k-1}u_k)\}$.
We claim that
\begin{equation}\label{S2Eq1}
S'_1,S'_k\subseteq V(P), \mbox{i.e., } S'_1\cup S'_k\subseteq V(P).\end{equation}

Suppose to the contrary that (\ref{S2Eq1}) is invalid.
Thus, at least one of $S'_1\setminus V(P)$ and $S'_k\setminus V(P)$ is nonempty.
By symmetry, we suppose $S'_k\setminus V(P)\neq \emptyset$
and let $u\in S'_k\setminus V(P)$.
If $u\in V(G\setminus H)$,
then $H+u_ku$ is a new $1^{(t)}$-path-cycle of $G$ such that $|H+u_ku|>|H|$,
contradicting the choice of $H$.
Thus, there must be a cycle $C$ of $H\setminus P$ such that $u\in V(C)$.
Since $C$ is a PC cycle, the two edges incident to $u$ in $C$
have distinct colors in $G$.
So we may assume that $C=ux_1x_2\ldots x_au$ such that $c(ux_1)\neq c(uu_k)$.
As a result, $P'=P\cup (C-x_au)= u_1u_2\ldots u_kux_1x_2\ldots x_a$ is a PC path of $G$.
Now $H'= P'\cup (H\setminus (P\cup C))$ is a new $1^{(t)}$-path-cycle of $G$
in which $P'$ is the unique path. But then $|H'|=|H|$ and $|P'|> |P|$, contradicting the maximality of $|P|$. This gives (\ref{S2Eq1}).

From (\ref{S2Eq1}), it follows that  $k=|P|\geq 2\delta^c(G)\geq \frac{4}{3}n+2t$.
Define $V'=\{u_1,u_2,\ldots,u_{t-1}\}\cup \{u_{k-1},u_{k-2},\ldots,u_{k-t+1}\}$.
Then set $S_1=S'_1\setminus V'$ and $S_2=S'_2\setminus V'$.
Now we introduce some notation used in the subsequent discussion.
Let $u$ be a vertex of $P$.
If $u=u_i$ for some $i\in\{2,3,\ldots,k\}$, then we define $u^-=u_{i-1}$.
If $u=u_i$ for some $i\in\{1,2,\ldots,k-1\}$, then we define $u^+=u_{i+1}$.
Note that $u_1^-$ and $u_k^+$ have no definitions.
For any two distinct vertices  $u_i, u_j\in V(P)$ such that
$i<j$, we write $u_i\prec u_j$ or $u_j\succ u_i$. Moreover,
we define $xP^+y=x x^+ x^{++}\ldots y$ and $yP^-x=yy^- y^{--}\ldots x$ for $x\prec y$.
Let $G[H]=G[V(H)]$ be the subgraph of $G$ induced by $V(H)$.

In the following discussion, we consider $k$ by its parity separately. We first consider the case where $k$ is even.

\begin{lemma}\label{l1}
If $k$ is even, then $G[H]$ has a PC $2^{(t)}$-factor.
\end{lemma}

\begin{proof}
Suppose to the contrary that $G[H]$ has no  PC $2^{(t)}$-factor.
Since $P$ is the unique path in the $1^{(t)}$-path-cycle $H$, $G[P]$ certainly has no PC $2^{(t)}$-factor.
For each $i\in\{1,k\}$,
set
\begin{equation}\label{S2D1}
R_i^{(1)}=\{u\in V(P):u^-\in S_i \mathrm{~and~} c(u_iu^-)\neq c(u^-u^{--})\}\end{equation}
and
\begin{equation}\label{S2D2}
R_i^{(2)}=\{u\in V(P):u^+\in S_i \mathrm{~and~} c(u_iu^+)\neq c(u^+u^{++})\}.\end{equation}
Recall from (\ref{S2Eq1}) that  $S_1, S_k\subseteq V(P)$.
By the definitions of $S_1$ and $S_k$, and the fact that $u_1,u_2\notin S_1'$,
we have $|S_1|\geq |S_1'|-\frac{2(t-1)-2}{2}\geq \frac{2}{3}n+1$.
In the same way, we derive $|S_k|\geq \frac{2}{3}n+1$.
Then we have
\begin{equation}\label{S2Eq4}
\left\{\begin{array}{l}
|R_1^{(1)}|+|R_1^{(2)}|\geq|S_1|\geq \frac{2}{3}n+1,\\[2mm]
|R_k^{(1)}|+|R_k^{(2)}|\geq|S_k|\geq \frac{2}{3}n+1.
\end{array}\right.
\end{equation}
Since $|P|=k$ is even and $G=K^c_{n,n}$ has bipartition $(X,Y)$,
without loss of generality, we may suppose that $u_1\in X$ and $u_k\in Y$.
From (\ref{S2D1}) and (\ref{S2D2}),  we have
\begin{equation}\label{S2Eq5}
\mbox{$R_1^{(1)}, R_1^{(2)}\subseteq X$ and $R_k^{(1)},R_k^{(2)}\subseteq Y$.}\end{equation}
Set
$$\left\{\begin{array}{ll}
X':=R_1^{(1)}\setminus R_1^{(2)},& X'':=R_1^{(2)}\setminus R_1^{(1)},\\[2mm]
Y':=R_k^{(1)}\setminus R_k^{(2)}, & Y'':=R_k^{(2)}\setminus R_k^{(1)},\\[2mm]
J_1:=R_1^{(1)}\cap R_1^{(2)}, & J_k:=R_k^{(1)}\cap R_k^{(2)},\\[2mm]
X^\star:= X'\cup X'', & Y^\star:= Y'\cup Y''.
\end{array}\right.$$
We then define a vertex coloring of the complete bipartite graph
$F:=(X^\star\cup J_1, Y^\star\cup J_k)$ in the following way.
For each vertex $u\in V(F)$, we define $c(u)$ by setting
\begin{equation}\label{c}
c(u)=\left\{
\begin{array}{ll}
c(uu^+),     &\mbox{if } {u\in X'\cup Y'},\\
c(uu^-),     &\mbox{if } {u\in X''\cup Y''},\\
c_0,   &\mbox{if } {u\in J_1\cup J_k} \ (\mbox{where $c_0$  is a new color}).
\end{array} \right.
\end{equation}
In the following we establish two useful claims.

\begin{claim}\label{c1}
For any two vertices  $x\in X^\star\cup J_1$ and $y\in Y^\star\cup J_k$,
if $d_P(x,y)\geq t-1$ or $d_P(x,y)=1$, then we have $c(xy)\in \{c(x),c(y)\}$.
As a consequence, either
$|J_1|\leq t-1$ or $|J_k|\leq t-1$.
\end{claim}

Suppose to the contrary that Claim \ref{c1} is invalid.
Then there exist two vertices $x\in X^\star\cup J_1$ and $y\in Y^\star\cup J_k$ such that $d_P(x,y)\geq t-1$ or $d_P(x,y)=1$, and $c(xy)\notin \{c(x),c(y)\}$.
Since $x\in X'\cup X''\cup J_1$ and $y\in Y'\cup Y''\cup J_k$,
we have the following nine cases to consider.

\vspace*{1mm}

\noindent{\bf Case 1.} $x\in X'$ and $y\in Y'$.
Then we have $c(u_1x^-)\neq c(x^-x^{--})$, $c(u_ky^{-})\neq c(y^{-}y^{--})$,
and $c(x)=c(xx^+)\neq c(xy)\neq c(yy^+)=c(y)$.
Furthermore, we have $xy\notin E(P)$
(for otherwise, we have $c(xy)\in \{c(xx^+),c(yy^+)\}=\{c(x),c(y)\}$, a contradiction).
If $x\prec y$, then the two cycles $C_1:=u_1P^+x^-u_1 $ and $C_2:=xP^+y^-u_kP^-yx$ are properly colored.
Since $|u_1P^+x^-|\geq t-1$ and $|y^-u_kP^-y|\geq t-1$,
we have $|C_1|\geq t$ and $|C_2|\geq t$.
In this case,
$C_1+C_2$ is a PC $2^{(t)}$-factor of $G[P]$, a contradiction.
If $x\succ y$, then the cycle $C_{3}:=u_1P^+y^-u_k P^-xyP^+x^-u_1$ with $|C_3|=|P|+1\geq t$.
In this case,
$C_3$ is
a PC $2^{(t)}$-factor of $G[P]$,
a contradiction again.

\vspace*{1mm}

\noindent{\bf Case 2.} $x\in X'$ and $y\in Y''$.
Then  $c(u_1x^-)\neq c(x^-x^{--})$, $c(u_ky^{+})\neq c(y^{+}y^{++})$,
and $c(x)=c(xx^+)\neq c(xy)\neq c(yy^-)=c(y)$.
If $xy\in E(P)$, then  we have  $x=y^+$ (for otherwise, $c(xy)=c(yy^-)=c(y)$, a contradiction). Then the two cycles $C_4:=u_1 P^+ y u_1$ and $C_5:=u_k P^- x u_k$ are properly colored.
Since $|u_1P^+y|\geq t-1$ and $|u_kP^-y|\geq t-1$,
we have $|C_4|\geq t$ and $|C_5|\geq t$.
In this case, $C_4+C_5$ is a PC $2^{(t)}$-factor of $G[P]$, a contradiction.
Hence, we have $xy\notin E(P)$.
If $x\prec y$, then the three cycles $C_{1}$,  $C_{6}:=xP^+yx$, and $C_{7}:=y^+P^+u_ky^+$ are properly colored.
Clearly, for every $i\in\{1,6,7\}$, the length of $C_i$ is at least $t$.
In this case,
$C_{1}+C_{6}+C_{7}$ is a PC $2^{(t)}$-factor of $G[P]$, a contradiction again.
If $x\succ y$, then the cycle $C_{8}:=u_1P^+yxP^+u_ky^+P^+x^-u_1$ is properly colored with $|C_8|=|P|+1>t$.
In this case, $C_{8}$ is a PC $2^{(t)}$-factor of $G[P]$, a contradiction still.

\vspace*{1mm}

\noindent{\bf Case 3.} $x\in X''$ and $y\in Y'$.
Then we have $c(u_1x^+)\neq c(x^+x^{++})$, $c(u_ky^{-})\neq c(y^{-}y^{--})$,
and $c(x)=c(xx^-)\neq c(xy)\neq c(yy^+)=c(y)$.
If $xy\in E(P)$, then we have  $x=y^-$
(for otherwise, $c(xy)=c(yy^+)=c(y)$, a contradiction).
Then the cycle $C_{9}:=u_1 P^+ x u_k P^- y u_1$  with $|C_9|=|P|+1>t$ is properly colored.
In this case, $C_9$ is a PC $2^{(t)}$-factor of $G[P]$, a contradiction.
Hence, we have $xy\notin E(P)$.
If $x\prec y$,
then the cycle $C_{10}:=u_1P^+xyP^+u_ky^-P^-x^+u_1$
is properly colored  with $|C_{10}|=|P|+1>t$.
In this case,  $C_{10}$ is a PC $2^{(t)}$-factor of $G[P]$, a contradiction again.
If $x\succ y$, then the two cycles $C_{11}:=u_1P^+y^-u_kP^-x^+u_1$ and $C_{12}:=yP^+xy$ are properly colored.
Clearly, for every $i\in\{11,12\}$, the length of $C_i$ is at least $t$.
In this case, $C_{11}+C_{12}$ is a PC $2^{(t)}$-factor of $G[P]$, a contradiction still.
\vspace*{1mm}

\noindent{\bf Case 4.} $x\in X''$ and $y\in Y''$.
Then we have $c(u_1x^+)\neq c(x^+x^{++})$, $c(u_ky^{+})\neq c(y^{+}y^{++})$,
and $c(x)=c(xx^-)\neq c(xy)\neq c(yy^-)=c(y)$.
Furthermore, we have $xy\notin E(P)$
(for otherwise, we have $c(xy)\in \{c(xx^-),c(yy^-)\}=\{c(x),c(y)\}$, a contradiction).
If $x\prec y$, then the two cycles $C_{7}$ and $C_{13}:=u_1P^+xyP^-x^+u_1$ are properly colored.
Clearly, for every $i\in\{7,13\}$, the length of $C_i$ is at least $t$.
In this case,  $C_{7}+C_{13}$ is a PC $2^{(t)}$-factor of $G[P]$, a contradiction.
If $x\succ y$, then the cycle $C_{14}:=u_1P^+yxP^+u_ky^+P^+x^-u_1$ is properly colored with $|C_{14}|=|P|+1>t$.
In this case, $C_{14}$ is a PC $2^{(t)}$-factor of $G[P]$, a contradiction again.

\vspace*{1mm}

\noindent{\bf Case 5.} $x\in X'$ and $y\in J_k$.
Then we have $c(u_1x^-)\neq c(x^-x^{--})$, $c(u_ky^{+})\neq c(y^{+}y^{++})$, $c(u_ky^{-})\neq c(y^{-}y^{--})$, and $c(x)=c(xx^+)\neq c(xy)$.
If $xy\in E(P)$, then  we have  $x=y^+$
(for otherwise, $c(xy)=c(xx^+)=c(x)$, a contradiction).
In this case, $C_4+C_5$ is a PC $2^{(t)}$-factor of $G[P]$, a contradiction.
Hence, we have $xy\notin E(P)$.
If  $c(xy)\neq c(yy^-)$,
then $C_1+C_6+C_7$ or $C_8$ is a PC $2^{(t)}$-factor of $G[P]$, a contradiction.
If  $c(xy)\neq c(yy^+)$,
then $C_1+C_2$ or $C_3$ is a PC $2^{(t)}$-factor of $G[P]$, a contradiction again.

\vspace*{1mm}

\noindent{\bf Case 6.} $x\in X''$ and $y\in J_k$.
Then we have $c(u_1x^+)\neq c(x^+x^{++})$, $c(u_ky^{+})\neq c(y^{+}y^{++})$, $c(u_ky^{-})\neq c(y^{-}y^{--})$, and $c(x)=c(xx^-)\neq c(xy)$.
If $xy\in E(P)$, then  we have  $x=y^-$
(for otherwise, $c(xy)=c(xx^-)=c(x)$, a contradiction).
In this case, $C_9$ is a PC $2^{(t)}$-factor of $G[P]$, a contradiction.
Hence, we have $xy\notin E(P)$.
If  $c(xy)\neq c(yy^-)$,
then $C_7+C_{13}$ or $C_{14}$ is a PC $2^{(t)}$-factor of $G[P]$, a contradiction.
If  $c(xy)\neq c(yy^+)$,
then $C_{10}$ or $C_{11}+C_{12}$ is a PC $2^{(t)}$-factor of $G[P]$, a contradiction again.

\vspace*{1mm}

\noindent{\bf Case 7.} $x\in J_1$ and $y\in Y'$.
Then we have $c(u_1x^{-})\neq c(x^{-}x^{--})$, $c(u_1x^{+})\neq c(x^{+}x^{++})$,
$c(u_ky^{-})\neq c(y^{-}y^{--})$,  and $c(y)=c(yy^+)\neq c(xy)$.
If $xy\in E(P)$, then we have $x=y^-$
(for otherwise, $c(xy)=c(yy^+)=c(x)$, a contradiction).
In this case, $C_9$ is a PC $2^{(t)}$-factor of $G[P]$, a contradiction.
Hence, we have $xy\notin E(P)$.
If  $c(xy)\neq c(xx^-)$,
then $C_{10}$ or $C_{11}+C_{12}$ is a PC $2^{(t)}$-factor of $G[P]$, a contradiction.
If  $c(xy)\neq c(xx^+)$,
then $C_1+C_2$ or $C_3$  is a PC $2^{(t)}$-factor of $G[P]$, a contradiction again.

\vspace*{1mm}

\noindent{\bf Case 8.} $x\in J_1$ and $y\in Y''$.
Then we have $c(u_1x^{-})\neq c(x^{-}x^{--})$, $c(u_1x^{+})\neq c(x^{+}x^{++})$,
$c(u_ky^{+})\neq c(y^{+}y^{++})$,  and $c(y)=c(yy^-)\neq c(xy)$.
If $xy\in E(P)$, then we have $x=y^+$, so $C_4+C_5$ is a PC 2-factor of $G[P]$, a contradiction.
Hence, we have $xy\notin E(P)$.
If  $c(xy)\neq c(xx^-)$,
then $C_7+C_{13}$ or $C_{14}$ is a PC $2^{(t)}$-factor of $G[P]$, a contradiction.
If  $c(xy)\neq c(xx^+)$,
then $C_1+C_6+C_7$ or $C_8$ is a PC $2^{(t)}$-factor of $G[P]$, a contradiction again.

\vspace*{1mm}

\noindent{\bf Case 9.} $x\in J_1$ and $y\in J_k$.
Suppose that $xy\in E(P)$.
If $x=y^+$, then  $C_4+C_5$ is a PC $2^{(t)}$-factor of $G[P]$, a contradiction.
If $x=y^-$, then $C_9$ is a PC $2^{(t)}$-factor of $G[P]$, a contradiction.
Thus, we have $xy\notin E(P)$.
If  $c(xx^-) \neq c(xy)\neq c(yy^-)$,
then  $C_7+C_{13}$ or $C_{14}$ is a PC $2^{(t)}$-factor of $G[P]$, a contradiction.
If  $c(xx^-) \neq c(xy)\neq c(yy^+)$,
then  $C_{10}$ or $C_{11}+C_{12}$ is a PC $2^{(t)}$-factor of $G[P]$, a contradiction.
If  $c(xx^+) \neq c(xy)\neq c(yy^-)$,
then $C_1+C_6+C_7$ or $C_8$  is a PC $2^{(t)}$-factor of $G[P]$, a contradiction.
If  $c(xx^+) \neq c(xy)\neq c(yy^+)$,
then  $C_1+C_2$ or $C_3$ is a PC $2^{(t)}$-factor of $G[P]$, a contradiction.
The claim holds.

\vspace*{1mm}

Note that $|X^\star|+2|J_1|=|R_1^{(1)}|+|R_1^{(2)}|\geq \frac{2}{3}n+1$
and $|Y^\star|+2|J_k|=|R_k^{(1)}|+|R_k^{(2)}|\geq \frac{2}{3}n+1$.
So, by $t\leq \frac{1}{3}n$,  we have
\begin{equation}\label{s2eq1}
|X^\star|\geq 1 \mathrm{~or~} |Y^\star|\geq 1.
\end{equation}

We now define a directed  bipartite graph $D$ with bipartition
$(X^\star\cup J_1, Y^\star)$ by setting the set of arcs
$A(D):= \{(x,y): x\in X^\star\cup J_1, d_P(x,y)\geq t-1 \mathrm{~or~} d_P(x,y)=1, y\in Y^\star\cup J_k, c(xy)\neq c(x)\}\cup
\{(y,x):  x\in X^\star\cup J_1, d_P(x,y)\geq t-1\mathrm{~or~}d_P(x,y)=1, y\in Y^\star\cup J_k, c(xy)\neq c(y)\}$.
By Claim \ref{c1},
we know that
$(y,x)\notin A(D)$ for every $y\in Y^\star$ and $x\in J_1$,
and
$(x,y)\notin A(D)$ for every $x\in X^\star$ and $y\in J_k$.
We further have the following claim for the digraph $D$.

\begin{claim}\label{c2}
$D$ has a directed 2-cycle or, equivalently,
there are two vertices $x$ and $y$ in $D$
such that $d_P(x,y)\geq t-1$ or $d_P(x,y)=1$,
and both $(x,y)$ and $(y,x)$ are arcs of $D$.
\end{claim}

From (\ref{S2Eq4}),
we have $2|X^\star\cup J_1|=2|X_1|+2|X_2|+2|J_1|\geq |X^\star|+|R^{(1)}_1|+|R^{(2)}_1|
\geq |X^\star|+\frac{2}{3}n+1$
and $2|Y^\star\cup J_k|=2|Y_1|+2|Y_2|+2|J_k|\geq |Y^\star|+|Q^{(1)}_k|+|Q^{(2)}_k|
\geq |Y^\star|+\frac{2}{3}n+1$.
Note that $|\{y\in V(P): 3\leq d_P(x,y)< t-1 \}|\leq t-2$ for each $x\in V(P)$.
Then we have
\begin{eqnarray*}
d_D^+(x)&=&|\{y\in N_G(x):c(xy)\neq c(x)\}\cap \{Y^\star\cup J_k\}|-(t-2)\\
          &\geq& |\{y\in N_G(x):c(xy)\neq c(x)\}|+|Y^\star|-n-(t-2)\\
          &\geq& \delta^c(G)-1+\frac{|Y^\star|+\frac{2}{3}n+1}{2}-n-(t-2)\\
          &\geq& \frac{|Y^\star|+3}{2}.
\end{eqnarray*}In the same way, we have
\begin{eqnarray*}
d_D^+(y)&=&|\{w:c(yw)\neq c(y)\}\cap (X^\star\cup J_1)|-(t-2)\\
          &\geq& d^c(x)-1+| X^\star\cup J_1|-n-(t-2)\\
          &\geq& \delta^c(G)-1+\frac{|X^\star|+\frac{2}{3}n+1}{2}-n-(t-2)\\
          &\geq& \frac{|Y^\star|+3}{2}.
\end{eqnarray*}
Since the arc from $ X^\star$ to $J_k$ does not exist, at least $|X^\star|\frac{|Y^\star|+3}{2}$ arcs are from $X^\star$ to $Y^\star$, i.e., $|A(X^\star,Y^\star)|\geq |X^\star|\frac{|Y^\star|+3}{2}$.
Since the arc from $ Y^\star$ to $J_1$ does not exist,
at least $|Y^\star|\frac{|X^\star|+3}{2}$ arcs are from $Y^\star$ to $X^\star$,
i.e., $|A(Y^\star,X^\star)|\geq |Y^\star|\frac{|X^\star|+3}{2}$.
Note that at most $|X^\star||Y^\star|$ arcs are not contained in any directed 2-cycle in $D[X^\star\cup Y^\star]$. By simple calculation,
we conclude that
$$|A(X^\star,Y^\star)|+|A(Y^\star,X^\star)|-|X^\star||Y^\star|\geq \frac{3}{2}(|X^\star|+|Y^\star|).$$
This means that at least $\frac{3}{2}(|X^\star|+|Y^\star|)$ arcs are contained in directed 2-cycles in $D[X^\star\cup Y^\star]$.
By (\ref{s2eq1}), we have $|X^\star|+|Y^\star|>0$.
The claim holds.

\vspace*{1mm}

By Claim \ref{c2}, there exist two vertices
$x\in X^\star\cup J_1$ and $y\in Y^\star\cup J_k$
such that $d_P(x,y)\geq k-1$ or $d_P(x,y)=1$
and $c(xy)\notin \{c(x),c(y)\}$,
which contradicts Claim \ref{c1}. This gives Lemma \ref{l1}.
\end{proof}

Lemma \ref{l1} can lead to a paradoxical result.
For the case where $k=|P|$ is even, from Lemma \ref{l1},
$G[H]$ has a PC $2^{(t)}$-factor $H'$.
Since $G=K_{n,n}^c$ has no PC $2^{(t)}$-factor and $H'$ consists of even cycles,
$G\setminus H'$ has at least one edge, denoted by $xy$.
But then $xy \cup H'$ is a $1^{(t)}$-path-cycle of $G$ larger than $H$,
contradicting the choice of $H$. Therefore, $k=|P|$ must be odd.
This further implies that $u_1u_k\notin E(G)$.

Let us keep the previous notation
\begin{equation}\label{s2eq7}
\left\{\begin{array}{l}
R_1^{(1)}=\{u\in V(P):u^-\in S_1 \mathrm{~and~} c(u_1u^-)\neq c(u^-u^{--})\},\\[2mm]
R_1^{(2)}=\{u\in V(P):u^+\in S_1 \mathrm{~and~} c(u_1u^+)\neq c(u^+u^{++})\}.
\end{array}\right.
\end{equation}
Moreover, we set
\begin{equation}\label{s2eq8}
\left\{\begin{array}{l}
Q_k^{(1)}=\{u\in V(P):u^{--}\in S_k \mathrm{~and~} c(u_ku^{--})\neq c(u^{--}u^{---})\},\\[2mm]
Q_k^{(2)}=\{u\in V(P):u^{++}\in S_k \mathrm{~and~} c(u_ku^{++})\neq c(u^{++}u^{+++})\}.
\end{array}\right.
\end{equation}
Recall from (\ref{S2Eq1}) that  $S_1, S_k\subseteq V(P)$.
By the definitions of $S_1$ and $S_k$,
we have $u_1,u_2,u_3,u_k\notin S_1'$.
This further implies that
$$|S_1|\geq |S_1'|-\lceil\frac{t-4}{2}\rceil-\lceil\frac{t-2}{2}\rceil\geq \frac{2}{3}n+1.$$
In the same way, we have  $|S_k|\geq \frac{2}{3}n+1$.

Since $u_1,u_k\notin S_1$ and $c(u_1u)\neq c(uu^-)$ or $c(u_1u)\neq c(uu^+)$ for any $u\in S_1$, we have $u^+\in R_1^{(1)}$ or $u^-\in R_1^{(2)}$ from (\ref{s2eq7}).
Since $u_1,u_2,u_{k-1},u_k\notin S_k$ and
$c(u_ku)\neq c(uu^-)$ or $c(u_ku)\neq c(uu^+)$ for any $u\in S_k$,
 we have $u^{++}\in R_1^{(1)}$ or $u^{--}\in R_1^{(2)}$ from (\ref{s2eq8}).
Thus, we have
\begin{equation}\label{S2Eq9}
\left\{\begin{array}{l}
|R_1^{(1)}|+|R_1^{(2)}|\geq|S_1|\geq \frac{2}{3}n+1,\\[2mm]
|Q_k^{(1)}|+|Q_k^{(2)}|\geq|S_k|\geq \frac{2}{3}n+1.
\end{array}\right.
\end{equation}

Note that either $u_1,u_k\in X$ or $u_1,u_k\in Y$.
By symmetry, we may suppose that $u_1,u_k\in X$.
Since $G=K^c_{n,n}$,  $k=|P|$ is odd and all the cycles in $H$ are even,
there must be some vertex $y^*\in Y\setminus V(H)$.
Since $G=K^c_{n,n}$ has bipartition $(X,Y)$,  from (\ref{s2eq7}) and  (\ref{s2eq8}),
we have
\begin{equation}
\mbox{$R_1^{(1)}\cup R_1^{(2)}\subseteq X$ and $Q_k^{(1)}\cup Q_k^{(2)}\subseteq Y$.}\end{equation}
Set
$$\left\{\begin{array}{ll}
X_1:=R_1^{(1)}\setminus R_1^{(2)},& X_2:=R_1^{(2)}\setminus R_1^{(1)},\\[2mm]
Y_1:=Q_k^{(1)}\setminus Q_k^{(2)}, & Y_2:=Q_k^{(2)}\setminus Q_k^{(1)},\\[2mm]
J':=R_1^{(1)}\cap R_1^{(2)}, & J'':=Q_k^{(1)}\cap Q_k^{(2)},\\[2mm]
X^\star:= X_1\cup X_2, & Y^\star:= Y_1\cup Y_2.
\end{array}\right.$$

We then define a vertex coloring of the complete bipartite graph
$F':=(X^\star\cup J', Y^\star\cup J'')$ in the following way.
For each vertex $u\in V(F')$, we define $c(u)$ by setting
\begin{equation}\label{c}
c(u)=\left\{
\begin{array}{ll}
c(uu^+),     &\mbox{if } {u\in X_1\cup Y_1},\\
c(uu^-),     &\mbox{if } {u\in X_2\cup Y_2},\\
c_0,   &\mbox{if } {u\in J'\cup J''} \ (\mbox{where $c_0$  is a new color}).
\end{array} \right.
\end{equation}
In the following we establish two useful claims.

\begin{claim}\label{c3}
For any two vertices  $x\in X^\star\cup J'$ and $y\in Y^\star\cup J''$,
we have $c(xy)\in \{c(x),c(y)\}$
if at least one of the following holds:

(a) $d_P(x,y)\neq 3$,

(b) $t\geq 5$ and $d_P(x,y)\geq t-1$,

(c) $d_P(x,y)=1$.

As a consequence, either $|J'|\leq t-1$ or $|J''|\leq t-1$.
\end{claim}

To prove Claim \ref{c3}, we suppose that there exist two vertices $x\in X^\star\cup J'$ and $y\in Y^\star\cup J''$ such that $c(xy)\notin \{c(x),c(y)\}$. Then it suffices to show that $d_P(x,y)=3$,  or $3\leq d_P(x,y)< t-1$ and $t\geq 5$.
We divide the discussion into the following nine cases.

\vspace*{1mm}

\noindent{\bf Case 1.} $x\in X_1$ and $y\in Y_1$.
Then we have $c(u_1x^-)\neq c(x^-x^{--})$, $c(u_ky^{--})\neq c(y^{--}y^{---})$,
and $c(x)=c(xx^+)\neq c(xy)\neq c(yy^+)=c(y)$.
Furthermore, we have $xy\notin E(P)$
(for otherwise, $c(xy)\in \{c(xx^+),c(yy^+)\}=\{c(x),c(y)\}$,
a contradiction).
If $x\prec y$, then the two cycles $C_1:=u_1 P^+ x^- u_1$ and $C_2:=xyP^+u_ky^{--}P^-x$ are properly colored.
Since $|u_1P^+x^-|\geq t-1$ and $|xyP^+u_ky^{--}|\geq t-1$,
we have $|C_1|\geq t$ and $|C_2|\geq t$.
Thus, $H_1:=H-P+C_1+C_2+y^-y^*$ is a $1^{(t)}$-path-cycle of $G$ larger than $H$,
a contradiction.
If $x\succ y$, then the cycle $C_3:=u_1P^+y^{--}u_kP^-xyP^+x^-u_1$ is properly colored with $|C_3|=k-1> t$.
Thus, $H_2:=H-P+C_3+y^-y^*$ is a $1^{(t)}$-path-cycle of $G$ larger than $H$,
a contradiction.

\vspace*{1mm}

\noindent{\bf Case 2.} $x\in X_1$ and $y\in Y_2$.
Then we have $c(u_1x^-)\neq c(x^-x^{--})$, $c(u_ky^{++})\neq c(y^{++}y^{+++})$,
and $c(x)=c(xx^+)\neq c(xy)\neq c(yy^-)=c(y)$.
If $xy\in E(P)$, then  we have  $x=y^+$
(for otherwise, $c(xy)=c(yy^-)=c(y)$,
a contradiction).
Then the two cycles $C_4:=u_1 P^+ y u_1$ and $C_5:=x^+ P^+ u_k x^+$ are properly colored.
Since $|u_1P^+y|\geq t-1$ and $|y^{++} P^+ u_k|\geq t-1$,
we have $|C_4|\geq t$ and $|C_5|\geq t$.
In this case,  $H_3:=H-P+C_4+C_5+xy^*$ is a $1^{(t)}$-path-cycle of $G$ larger than $H$,
a contradiction. Hence, we must have $xy\notin E(P)$.
If $x\prec y$, then the three cycles $C_1$, $C_6:=xP^+yx$, and $C_7:=y^{++}P^+u_k y^{++}$ are properly colored.
Since $|u_1P^+y|\geq t-1$ and $|y^{++} P^+ u_k|\geq t-1$,
we have $|C_1|\geq t$ and $|C_7|\geq t$.
Thus,  $H_4:=H-P+C_1+C_6+C_7+y^+y^*$ is a $1^{(t)}$-path-cycle of $G$ larger than $H$,
unless $|C_6|< t$.
This further implies $d_P(x,y)< t-1$ and $t\geq 5$.
Since $xy\notin E(P)$, we have $3\leq d_P(x,y)< t-1$ and $t\geq 5$.
If $x\succ y$ and $d_P(x,y)\neq 3$ (this implies $y^{++}\neq x^-$),
then the cycle $C_8:=u_1 P^+ y x P^+u_k y^{++} P^+ x^- u_1$ is properly colored
with $|C_8|=k-1> t$.
Thus,  $H_5:=H-P+C_8+y^+y^*$ is a $1^{(t)}$-path-cycle of $G$ larger than $H$,
a contradiction.
Hence, we have $d_P(x,y)=3$, or $3\leq d_P(x,y)< t-1$ and $t\geq 5$.
\vspace*{1mm}

\noindent{\bf Case 3.} $x\in X_2$ and $y\in Y_1$.
Then we have $c(u_1x^+)\neq c(x^+x^{++})$, $c(u_ky^{--})\neq c(y^{--}y^{---})$,
and $c(x)=c(xx^-)\neq c(xy)\neq c(yy^+)=c(y)$.
If $xy\in E(P)$, then we have  $x=y^-$
(for otherwise, we have $c(xy)=c(yy^+)=c(y)$,
a contradiction).
Then the cycle $C_9:=u_1 P^+ y^{--}u_kP^- x^+ u_1$ is properly colored
with $|C_9|=k-1> t$.
In this case,  $H_6:=H-P+C_9+xy^*$ is a $1^{(t)}$-path-cycle of $G$ larger than $H$,
a contradiction.
Hence, we must have $xy\notin E(P)$.
If $x\prec y$ and $d_P(x,y)\neq 3$ (this implies $y^{--}\neq x^+$),
then the cycle $C_{10}:=u_1 P^+ xy P^+u_k y^{--} P^- x^+ u_1$ is properly colored
with $|C_{10}|=k-1> t$.
Thus,  $H_7:=H-P+C_{10}+y^-y^*$ is a $1^{(t)}$-path-cycle of $G$ larger than $H$,
a contradiction.
If $x\succ y$, then the two cycles $C_{11}:=u_1 P^+ y^{--} u_kP^-x^+u_1$ and $C_{12}:=yP^+xy$ are properly colored.
Clearly, the length of $C_{11}$ is at least $t$.
Thus, $H_8:=H-P+C_{11}+C_{12}+y^-y^*$ is a $1^{(t)}$-path-cycle of $G$ larger than $H$, unless $|C_{12}|< t$.
This further implies $d_P(x,y)< t-1$ and $t\geq 5$.
Since $xy\notin E(P)$, we have $3\leq d_P(x,y)< t-1$ and $t\geq 5$.
Hence, we have $d_P(y,x)=3$, or $3\leq d_P(x,y)< t-1$ and $t\geq 5$.
\vspace*{1mm}

\noindent{\bf Case 4.} $x\in X_2$ and $y\in Y_2$.
Then we have $c(u_1x^+)\neq c(x^+x^{++})$, $c(u_ky^{++})\neq c(y^{++}y^{+++})$,
and $c(x)=c(xx^-)\neq c(xy)\neq c(yy^-)=c(y)$.
Furthermore, we have $xy\notin E(P)$
(for otherwise, we have $c(xy)\in \{c(xx^-),c(yy^-)\}=\{c(x),c(y)\}$,
a contradiction).
If $x\prec y$, then the two cycles
$C_{13}:=u_1P^+xyP^-x^+u_1$ and $C_{14}:=y^{++}P^+u_k y^{++}$ are properly colored.
Clearly, we have $|C_{13}|\geq t$ and $|C_{14}|\geq t$.
Thus, $H_9:=H-P+C_{13}+C_{14}+y^+y^*$ is a $1^{(t)}$-path-cycle of $G$ larger than $H$,
a contradiction.
If $x\succ y$,
then the cycle $C_{15}:=u_1P^=yxP^-y^{++}u_kP^-x^+u_1$ is properly colored
with $|C_{15}|=k-1> t$.
Thus,  $H_{10}:=H-P+C_{15}+y^+y^*$ is a $1^{(t)}$-path-cycle of $G$ larger than $H$,
a contradiction.
\vspace*{1mm}

\noindent{\bf Case 5.} $x\in X_1$ and $y\in J''$.
Then we have $c(u_1x^-)\neq c(x^-x^{--})$, $c(u_ky^{++})\neq c(y^{++}y^{+++})$, $c(u_ky^{--})\neq c(y^{--}y^{---})$, and $c(x)=c(xx^+)\neq c(xy)$.
If $xy\in E(P)$, then  we have  $x=y^+$
(for otherwise, we have $c(xy)=c(xx^+)=c(x)$,
a contradiction).
Then  $H_3$ is a $1^{(t)}$-path-cycle of $G$ larger than $H$,
a contradiction.
If  $c(xy)\neq c(yy^-)$,
then $H_4$ or $H_5$ is a $1^{(t)}$-path-cycle of $G$ larger than $H$,
unless $3\leq d_P(x,y)< t-1$ and $t\geq 5$.
If  $c(xy)\neq c(yy^+)$,
then  $H_1$ or $H_2$ is a $1^{(t)}$-path-cycle of $G$ larger than $H$, unless $d_P(x,y)= 3$,
a contradiction.
Hence, we have $d_P(x,y)=3$, or $3\leq d_P(x,y)< t-1$ and $t\geq 5$.

\vspace*{1mm}

\noindent{\bf Case 6.} $x\in X_2$ and $y\in J''$.
Then we have $c(u_1x^-)\neq c(x^-x^{--})$, $c(u_ky^{++})\neq c(y^{++}y^{+++})$, $c(u_ky^{--})\neq c(y^{--}y^{---})$, and $c(x)=c(xx^+)\neq c(xy)$.
If $xy\in E(P)$, then  we have  $x=y^-$
(for otherwise, we have $c(xy)=c(xx^-)=c(x)$,
a contradiction).
Then  $H_6$ is a $1^{(t)}$-path-cycle of $G$ larger than $H$,
a contradiction.
If  $c(xy)\neq c(yy^-)$,
then $H_9$ or $H_{10}$ is a $1^{(t)}$-path-cycle of $G$ larger than $H$,
a contradiction.
If  $c(xy)\neq c(yy^+)$,
then $H_7$  is a $1^{(t)}$-path-cycle of $G$ larger than $H$,
unless $d_P(x,y)= 3$,
and $H_8$  is a $1^{(t)}$-path-cycle of $G$ larger than $H$,
unless $3\leq d_P(x,y)< t-1$ and $t\geq 5$.
Hence, we have $d_P(x,y)=3$, or $3\leq d_P(x,y)< t-1$ and $t\geq 5$.

\vspace*{1mm}

\noindent{\bf Case 7.} $x\in J'$ and $y\in Y_1$.
Then we have $c(u_1x^{-})\neq c(x^{-}x^{--})$, $c(u_1x^{+})\neq c(x^{+}x^{++})$,
$c(u_ky^{--})\neq c(y^{--}y^{---})$,  and $c(y)=c(yy^+)\neq c(xy)$.
If $xy\in E(P)$, then we have $x=y^-$
(for otherwise, we have $c(xy)=c(yy^+)=c(y)$,
a contradiction).
Then $H_6$ is a $1^{(t)}$-path-cycle of $G$ larger than $H$,
a contradiction.
If  $c(xy)\neq c(xx^-)$,
then $H_7$  is a $1^{(t)}$-path-cycle of $G$ larger than $H$,
unless $d_P(x,y)= 3$,
and $H_8$  is a $1^{(t)}$-path-cycle of $G$ larger than $H$,
unless $3\leq d_P(x,y)< t-1$ and $t\geq 5$.
If  $c(xy)\neq c(xx^+)$,
then $H_1$ or $H_2$ is a $1^{(t)}$-path-cycle of $G$ larger than $H$,
a contradiction.
Hence, we have $d_P(x,y)=3$, or $3\leq d_P(x,y)< t-1$ and $t\geq 5$.
\vspace*{1mm}

\noindent{\bf Case 8.} $x\in J'$ and $y\in Y_2$.
Then we have $c(u_1x^{-})\neq c(x^{-}x^{--})$, $c(u_1x^{+})\neq c(x^{+}x^{++})$,
$c(u_ky^{++})\neq c(y^{++}y^{+++})$,  and $c(y)=c(yy^-)\neq c(xy)$.
If $xy\in E(P)$, then we have $x=y^+$
(for otherwise, we have $c(xy)=c(yy^-)=c(y)$,
a contradiction).
Then  $H_3$ is a $1^{(t)}$-path-cycle of $G$ larger than $H$,
a contradiction.
If  $c(xy)\neq c(xx^-)$,
then  $H_9$ or $H_{10}$  is a $1^{(t)}$-path-cycle of $G$ larger than $H$,
unless $d_P(x,y)= 3$,
a contradiction.
If  $c(xy)\neq c(xx^+)$,
then $H_5$  is a $1^{(t)}$-path-cycle of $G$ larger than $H$,
unless $d_P(x,y)= 3$,
and $H_4$  is a $1^{(t)}$-path-cycle of $G$ larger than $H$
unless $3\leq d_P(x,y)< t-1$ and $t\geq 5$.
Hence, we have $d_P(x,y)=3$, or $3\leq d_P(x,y)< t-1$ and $t\geq 5$.

\vspace*{1mm}

\noindent{\bf Case 9.} $x\in J'$ and $y\in J''$.
Suppose first that $xy\in E(P)$.
If $x=y^+$, then  $H_3$ is a $1^{(t)}$-path-cycle of $G$ larger than $H$,
a contradiction.
If $x=y^-$, then  $H_1$ is a $1^{(t)}$-path-cycle of $G$ larger than $H$,
a contradiction.
Hence, we must have $xy\notin E(P)$.
If  $c(xx^-) \neq c(xy)\neq c(yy^-)$,
then  $H_9$ or $H_{10}$ is a $1^{(t)}$-path-cycle of $G$ larger than $H$,
a contradiction.
If  $c(xx^-) \neq c(xy)\neq c(yy^+)$,
then $H_7$  is a $1^{(t)}$-path-cycle of $G$ larger than $H$,
unless $d_P(x,y)= 3$,
and $H_8$  is a $1^{(t)}$-path-cycle of $G$ larger than $H$,
unless $3\leq d_P(x,y)< t-1$ and $t\geq 5$.
If  $c(xx^+) \neq c(xy)\neq c(yy^-)$,
then $H_5$  is a $1^{(t)}$-path-cycle of $G$ larger than $H$
unless $d_P(x,y)= 3$,
and $H_4$  is a $1^{(t)}$-path-cycle of $G$ larger than $H$
unless $3\leq d_P(x,y)< t-1$ and $t\geq 5$.
If  $c(xx^+) \neq c(xy)\neq c(yy^-)$,
then  $H_1$ or $H_{2}$ is a 1-path-cycle of $G$ larger than $H$,
a contradiction.
As a result,  we have $d_P(x,y)=3$, or $3\leq d_P(x,y)< t-1$ and $t\geq 5$.
The claim holds.
\vspace*{3mm}

Note that $|X^\star|+2|J'|=|R_1^{(1)}|+|R_1^{(2)}|\geq \frac{2}{3}n+1$
and $|Y^\star|+2|J''|=|R_k^{(1)}|+|R_k^{(2)}|\geq \frac{2}{3}n+1$.
So, by $t\leq \frac{1}{3}n$,  we have
\begin{equation}\label{s2equ2}
|X^\star|\geq 1 \mathrm{~or~} |Y^\star|\geq 1.
\end{equation}
Now we define two directed  bipartite graphs.
If $t\leq 4$,
then we define a directed  bipartite graph $D_1$ with bipartition
$( X^\star\cup J', Y^\star\cup J'')$ by setting the set of arcs
$A(D_1):= \{
(x,y): x\in X^\star\cup J', y\in Y^\star\cup J'', d_P(x,y)\neq 3, c(xy)\neq c(x)\}\cup
\{(y,x): x\in X^\star\cup J', y\in Y^\star\cup J'', d_P(x,y)\neq 3, c(xy)\neq c(y)\}.$

If $t\geq 5$,
then we define a directed  bipartite graph $D_2$ with bipartition
$( X^\star\cup J', Y^\star\cup J'')$ by setting the set of arcs
$A(D_2):= \{
(x,y): x\in X^\star\cup J', y\in Y^\star\cup J'', d_P(x,y)\geq t-1 \mathrm{~or~} d_P(x,y)=1, c(xy)\neq c(x)\}\cup
\{(y,x):  x\in X^\star\cup J', y\in Y^\star\cup J'', d_P(x,y)\geq t-1 \mathrm{~or~} d_P(x,y)=1, c(xy)\neq c(y)\}$.

Note that for any two vertices $x\in X^\star\cup J'$ and $y\in  J''$,
we have $(x,y)\notin A(D)$,
and for any two vertices $x\in  J'$ and $y\in  Y^\star\cup J''$,
we have $(y,x)\notin A(D)$.

\vspace*{1mm}

\vspace*{1mm}

\begin{claim}\label{c5}
Both $D_1$ and $D_2$ contain a directed 2-cycles or, equivalently, there are two vertices $x$ and $y$ in $D_1$
such that $d_P(x,y)\neq 3$ and both
$(x,y)$ and $(y,x)$ are arcs of $D_1$,
and there are two vertices $x'$ and $y'$ in $D_2$
such that $d_P(x',y')\geq t-1 \mathrm{~or~} d_P(x',y')=1$ and both
$(x',y')$ and $(y',x')$ are arcs of $D_2$.
\end{claim}

From (\ref{S2Eq9}),
we have $2|X^\star\cup J'|=2|X_1|+2|X_2|+2|J'|= |X^\star|+|R^{(1)}_1|+|R^{(2)}_1|
\geq |X^\star|+\frac{2}{3}n+1$
and $2|Y^\star\cup J''|=2|Y_1|+2|Y_2|+2|J''|= |Y^\star|+|Q^{(1)}_k|+|Q^{(2)}_k|
\geq |Y^\star|+\frac{2}{3}n+1$.
Since $|\{y\in V(P) :  d_P(x,y)= 3 \}|=2$
for every $x\in X^\star$,
by Claim \ref{c3}, we have
\begin{eqnarray*}
d_{D_1}^+(x)&\geq&|\{y:c(xy)\neq c(x)\}\cap (Y^\star\cup J'')|-2\\
          &\geq& d^c(x)-1+\frac{|Y^\star|+\frac{2}{3}n+1}{2}-n-2\\
          &\geq& \delta^c(G)-1+\frac{|Y^\star|+\frac{2}{3}n+1}{2}-n-2\\
          &\geq& \frac{|Y^\star|+1}{2}.
\end{eqnarray*}
Since $|\{y\in V(P) : 3\leq d_P(x,y)< t-1 \}|\leq t-2$
for every $x\in X^\star$,
by Claim \ref{c3}, we have
\begin{eqnarray*}
d_{D_2}^+(x)&\geq&|\{y:c(xy)\neq c(x)\}\cap (Y^\star\cup J'')|-(t-2)\\
          &\geq& d^c(x)-1+\frac{|Y^\star|+\frac{2}{3}n+1}{2}-n-(t-2)\\
          &\geq& \delta^c(G)-1+\frac{|Y^\star|+\frac{2}{3}n+1}{2}-n-(t-2)\\
          &\geq& \frac{|Y^\star|+3}{2}.
\end{eqnarray*}
In the same way, for every $y\in Y^\star$, we have
$$d_{D_1}^+(y)\geq \frac{|Y^\star|+1}{2} \mathrm{~and~} d_{D_2}^+(y)\geq \frac{|Y^\star|+3}{2}.$$
Since $(x,y)\notin A(D_1)$ and $(x,y)\notin A(D_2)$ for any two vertices $x\in X^\star\cup J'$ and $y\in  J''$,
there are at least $|X^\star|\frac{|Y^\star|+1}{2}$ arcs from $X^\star$ to $Y^\star$ in $D_i$ for each $i\in\{1,2\}$.
In the same way, for each $i\in\{1,2\}$,
there are at least   $|Y^\star|\frac{|X^\star|+1}{2}$ arcs from $Y^\star$ to $X^\star$ in $D_i$.
Since there are at most $|X^\star||Y^\star|$ arcs not contained in any directed 2-cycle in $X^\star\cup Y^\star$ and
$$A(X^\star,Y^\star)+A(Y^\star,X^\star)-|X^\star||Y^\star|\geq \frac{1}{2}(|X^\star|+|Y^\star|),$$
there  exist at least $\frac{1}{2}(|X^\star|+|Y^\star|)$ arcs  contained in  directed 2-cycles.
By (\ref{s2equ2}), we have $\frac{1}{2}(|X^\star|+|Y^\star|)>0$.
The claim holds.

\vspace*{3mm}
By Claim \ref{c5},
if $k\leq 4$,
then there exist two vertices $x\in X^\star$ and $y\in Y^\star$ such that $d_P(x,y)\neq 3$ and $c(xy)\notin \{c(x),c(y)\}$,
and if $k\geq 5$,
then there exist two vertices $x\in X^\star$ and $y\in Y^\star$ such that $d_P(x,y)\geq t-1\mathrm{~or~} d_P(x',y')=1$ and $c(xy)\notin \{c(x),c(y)\}$,
which contradicts Claim \ref{c3} in both situations.
This completes the proof of Theorem \ref{t1}.

\vspace*{1mm}
Note that, from Theorem \ref{t1}, we have the following useful corollary.
\begin{corollary}\label{coro2}
Let  $K^c_{n,n}$ be an edge-colored graph such that $\delta^c(K^c_{n,n})\geq \frac{2n}{3}+t$.
Then $K^c_{n,n}$ can be covered by at most $\lceil\frac{2n}{t}\rceil$ PC odd paths.
\end{corollary}

\section{Proof of  Theorem \ref{t2}}
Let $\varepsilon>0$ be a fixed number.
Let $G=K^c_{n,n}$ be an edge-colored complete balanced bipartite graph with bipartition $(X,Y)$ and $\delta^c(G)\geq (\frac{2}{3}+\varepsilon)n$.
This means that $\varepsilon \leq \frac{1}{3}$.
We first introduce two useful definitions,
where Definition \ref{S3D2} was first introduced in Lo \cite{A1}.

\begin{definition}\label{S3D1}
For every $x\in X$ and $y\in Y$,
a path $P$ of $G$ is called an \textbf{absorbing path} for $(x,y)$ if
the following statements hold.

(i) $P:=x'y'x''y''$ is a PC $3$-path;

(ii) $V(P)\cap \{x,y\}=\emptyset$;

(iii) the path $x'y'xyx''y''$ is properly colored.
\end{definition}

\begin{definition}\label{S3D2}
Let $x_1,x_2,y_1,y_2$ be four distinct vertices of $G$ with
$x_1,x_2\in X$ and $y_1,y_2\in Y$.
A path $P$ of $G$ is called an \textbf{absorbing path} for $(x_1,y_1;x_2,y_2)$ if
the following statements hold.

(i) $P:=x'y'x''y''$ is a PC $3$-path;

(ii) $V(P)\cap \{x_1,y_1, x_2,y_2\}=\emptyset$;

(iii) the two paths $x'y'x_1y_1$ and $y''x''y_2x_2$ are  properly colored.
\end{definition}

For convenience, each $(x,y)$ in Definition \ref{S3D1}
is called a \emph{D1-element} and each $(x_1,y_1; x_2,y_2)$ in Definition \ref{S3D2}
is called a \emph{D2-element}.
For each D1-element $(x,y)$, we use
$\mathcal{P}(x,y)$ to denote the set of absorbing paths for $(x,y)$,
and for each D2-element $(x_1,y_1; x_2,y_2)$,
we use $\mathcal{P}(x_1,y_1;x_2,y_2)$ to denote the set of absorbing paths for $(x_1,y_1;x_2,y_2)$.
Then all the paths in $\mathcal{P}(x,y)$
and $\mathcal{P}(x_1,y_1;x_2,y_2)$ are PC $3$-paths.
The following proposition can be found in Lo \cite{A1}.

\begin{prop}[\cite{A1}]\label{p0}
Let $P':=x_1y_1x_2y_2\ldots x_{l}y_l$ be a PC path with $l\geq 2$
and $P:=x'y'x''y''$ be an absorbing  path for $(x_1,y_1;x_{l},y_l)$.
Suppose that $V(P)\cap V(P')=\emptyset$.
Then $x'y'x_1y_1x_2y_2\ldots x_{l}y_lx''y''$ is a PC path.
\end{prop}

The next lemma establishes a lower bound on $|\mathcal{P}(x,y)|$
and a lower bound on $|\mathcal{P}(x_1,y_1;x_2,y_2)|$.

\begin{lemma}\label{s3l1}
Suppose that $n\geq \frac{4}{\varepsilon}$.
Then we have the following two inequalities:

(i) $|\mathcal{P}(x,y)|\geq \frac{16}{9}\varepsilon^2 n^4$ for every D1-element $(x,y)$, and

(ii) $|\mathcal{P}(x_1,y_1;x_2,y_2)|\geq \frac{16}{9}\varepsilon^2 n^4$
for every D2-element $(x_1,y_1;x_2,y_2)$.
\end{lemma}

\begin{proof}
Since $\delta^c(K^c_{n,n})+\Delta_{mon}(K^c_{n,n})\leq n+1$,
we have $\Delta_{mon}(K^c_{n,n})\leq (\frac{1}{3}-\varepsilon)n+1$. To prove (i), we consider an arbitrary absorbing path $P=x'y'x''y''$  for $(x,y)$.
Then $V(P)\cap \{x,y\}=\emptyset$.
Since $c(xy')\neq c(xy)$ and $c(x'y')\neq c(xy')$,
each of $x'$ and $y'$  has at least $(\frac{2}{3}+\varepsilon)n-1$ choices.
Since $n\geq \frac{4}{\varepsilon}$,
we have $(\frac{2}{3}+\varepsilon)n-1> \frac{2}{3} n$.
Since $c(x''y')\neq c(y'x')$ and $c(x''y)\neq c(x''y')$,
$x''$ has at least
$\delta^c(K^c_{n,n})-1-\Delta_{mon}(K^c_{n,n})\geq (\frac{1}{3}+2\varepsilon)n-2\geq  (\frac{1}{3}+\varepsilon)n$
choices, given $(x',y')$.
Since $c(y''x'')\neq c(x''y)$ and $c(y''x'')\neq c(x''y')$,
$y''$ has at least
$ \delta^c(K^c_{n,n})-1-\Delta_{mon}(K^c_{n,n})\geq (\frac{1}{3}+\varepsilon)n$
choices, given $(x',y',x'')$.
Hence, we have
$$|\mathcal{P}(x,y)|\geq
(\frac{2}{3}n)^2[(\frac{1}{3}+\varepsilon)n]^2
=\frac{4}{9}(\frac{1}{3}+\varepsilon)^2 n^4\geq \frac{4}{9}(2\varepsilon)^2 n^4\geq\frac{16}{9}\varepsilon^2 n^4.$$
This gives (i).

To prove (ii), we consider an arbitrary absorbing path
$P=x'y'x''y''$ for $(x_1,y_1;x_2,y_2)$.
Then $V(P)\cap \{x_1,y_1,x_2,y_2\}=\emptyset$.
Since $x',y'\notin\{x_1,y_1;x_2,y_2\}$,
$c(x_2y_2)\neq c(y_2x'')$, and $c(x_1y')\neq c(x_1y_1)$,
each of $y'$ and $x''$ has at least $(\frac{2}{3}+\varepsilon)n-2$ choices.
Since $n\geq \frac{4}{\varepsilon}$,
we have $(\frac{2}{3}+\varepsilon)n-2\geq \frac{2}{3} n$.
Since $x'\notin\{x_1, x_2 ,x''\}$, $c(x'y')\neq c(y'x'')$, and $c(x'y')\neq c(x_1y')$,
$x'$ has at least
$(\frac{2}{3}+\varepsilon)n-3-[(\frac{1}{3}-\varepsilon)n+1]
=(\frac{1}{3}+2\varepsilon)n-4$ choices, given $(y',x'')$.
Since $n\geq \frac{4}{\varepsilon}$,
we have $(\frac{1}{3}+2\varepsilon)n-4\geq (\frac{1}{3}+\varepsilon)n$.
In the same way,
 $y''$ has at least $(\frac{1}{3}+\varepsilon)n$ choices, given $(x',y',x'')$.
Hence, we have
$$|\mathcal{P}(x_1,y_1;x_2,y_2)|\geq (\frac{2}{3} n)^2[(\frac{1}{3}+\varepsilon)n]^2=\frac{4}{9}(\frac{1}{3}+\varepsilon)^2 n^4\geq \frac{4}{9}(2\varepsilon)^2 n^4\geq\frac{16}{9}\varepsilon^2 n^4.$$
This gives (ii) and the lemma follows.
\end{proof}

Next, we find a family $\mathcal{F}$ of vertex-disjoint PC $3$-paths
with some probabilistic arguments
such that, for every D2-element $(x_1,y_1;x_2,y_2)$,
the number of  absorbing  paths for $(x_1,y_1)$ in ${\cal F}$,
i.e., $|{\cal F}\cap \mathcal{P}(x,y)|$,
and the number of  absorbing  paths for $(x_1,y_1;x_2,y_2)$,
i.e., $|{\cal F}\cap \mathcal{P}(x_1,y_1;x_2,y_2)|$,
are linear with respect to $n$, respectively.
Finally,  we find a small cycle $C$ from $\mathcal{F}$ and use $C$
to absorb some PC paths outside  $C$.

Recall the famous Chernoff bound for the binomial distribution
and Markov's inequality, which we will use in the proof of Lemma \ref{s3l2}.

\begin{prop}\label{p1}
\textbf{(Chernoff bound)}
Suppose that $X$ has the binomial distribution and $0<a<3/2$.
Then
$\mathbf{Pr}(|X-\mathbf{E}(X)|\geq a\mathbf{E}(X))    \leq 2e^{-a^2\frac{\mathbf{E}(X)}{3}}$.
\end{prop}

\begin{prop}\label{p2}
\textbf{(Markov's inequality)}
Suppose that $Y$ is an arbitrary
nonnegative random variable and $a>0$. Then
$\mathbf{Pr}[Y> a\mathbf{E}(Y)]<\frac{1}{a}$.
\end{prop}

\begin{lemma}\label{s3l2}
Consider a D1-element $(x,y)$ and a D2-element $(x_1,y_1;x_2,y_2)$.
Let $0<\gamma<1/2$, and suppose that
$|\mathcal{P}(x,y)|\geq \gamma n^4$
and $|\mathcal{P}(x_1,y_1;x_2,y_2)|\geq \gamma n^4$.
Then there exists an integer $n_0(\gamma)$ such that,
whenever $n\geq n_0(\varepsilon)$, there exists a family $\mathcal{F}$ of vertex-disjoint PC $3$-paths such that
$$|\mathcal{F}|\leq 2^{-4}\gamma n,$$
$$|\mathcal{F}\cap \mathcal{P}(x,y)|\geq 2^{-7}\gamma^2 n, $$
and
$$|\mathcal{F}\cap \mathcal{P}(x_1,y_1;x_2,y_2)|\geq 2^{-7}\gamma^2 n.$$
\end{lemma}

\begin{proof}
Let us choose an integer $n_0=n_0(\gamma)$ sufficiently large so that
\begin{equation}\label{e1}
\max\{\exp\{-\gamma n_0/(3\times 2^5)\}, \exp\{-\gamma^2 n_0/(3\times 2^7)\}        \}\leq 1/6.
\end{equation}
Then we assume in the following that $n\geq n_0$.
We consider a randomly generated family $\mathcal{F'}$ of
3-paths $x'y'x''y''$.
There are totally $n^2(n-1)^2$ possible such 3-paths (candidates).
But $\mathcal{F'}$ is generated by selecting each candidate  independently
at random with probability
$p=2^{-5} \gamma \frac{1}{n(n-1)^2}\geq 2^{-5} \gamma n^{-3}$.
Then we have
$$\mathbf{E}(|\mathcal{F'}|)=pn^2(n-1)^2= 2^{-5} \gamma n,$$
$$\mathbf{E}(|\mathcal{F}'\cap \mathcal{P}(x,y)|)=p\gamma n^4 \geq 2^{-5} \gamma^2 n,$$
and
$$\mathbf{E}(|\mathcal{F'}\cap \mathcal{P}(x_1,y_1;x_2,y_2)|)=p\gamma n^4 \geq 2^{-5} \gamma^2 n.$$
By Proposition \ref{p1} and inequality (\ref{e1}), each of  the
following three inequalities
$$|\mathcal{F'}|\leq 2\mathbf{E}(|\mathcal{F'}|)=2^{-4} \gamma n,$$
$$|\mathcal{F'}\cap \mathcal{P}(x,y)|\geq \frac{1}{2}
 \mathbf{E}(|\mathcal{F'}\cap \mathcal{P}(x,y)|)\geq 2^{-6} \gamma^2 n$$,
$$|\mathcal{F'}\cap \mathcal{P}(x_1,y_1;x_2,y_2)|\geq \frac{1}{2}
 \mathbf{E}(|\mathcal{F'}\cap \mathcal{P}(x_1,y_1;x_2,y_2)|)\geq 2^{-6} \gamma^2 n$$
holds with a probability of at least  $\frac{2}{3}$.

We say that two 3-paths
$x'y'x''y''$ and $\tilde{x}'\tilde{y}'\tilde{x}''\tilde{y}''$ are \emph{intersecting}
if they have at least one common vertex, i.e.,
$\{x',y',x'',y''\}\cap \{\tilde{x}',\tilde{y}',\tilde{x}'',\tilde{y}''\}\neq \emptyset$.
Then the expected number of intersecting 3-path pairs in $\mathcal{F'}$
is about
$$n^2(n-1)^2\times 4^2 \times n(n-1)^2 \times p^2=2^{-6}\gamma^2 n.$$
By Proposition \ref{p2}, with a probability of at least 1/2,
$\mathcal{F'}$ contains at most $2^{-7}\gamma^2 n$ pairs of intersecting 3-paths.
Now we obtain a family $\mathcal{F}$ of vertex-disjoint PC $3$-paths from $\mathcal{F}'$
in the following way: first remove a 3-path from each pair of intersecting $3$-paths of $\mathcal{F'}$, and then delete all the $3$-paths that are not properly colored.
It follows that
$$|\mathcal{F}|\leq|\mathcal{F'}|\leq 2^{-4}\gamma n,$$
$$|\mathcal{F}\cap \mathcal{P}(x,y)|\geq 2^{-6} \gamma^2 n-2^{-7}\gamma^2 n= 2^{-7}\gamma^2 n,$$
and
$$|\mathcal{F}\cap \mathcal{P}(x_1,y_1;x_2,y_2)|\geq 2^{-6} \gamma^2 n-2^{-7}\gamma^2 n= 2^{-7}\gamma^2 n.$$
The result follows.
\end{proof}

\begin{lemma}\label{lem3}
Suppose that $\delta^c(K^c_{n,n})\geq \frac{2}{3}n+2$.
Then, for every D2-element $(x_1,y_1;x_2,y_2)$,
there exist at least $\frac{4}{3}n$ edges $xy$ of $G$
with $x\in X\setminus \{x_1,x_2\}$ and $y\in Y\setminus \{y_1,y_2\}$
such that $x_1y_1xyx_2y_2$ is a PC path.
\end{lemma}

\begin{proof}
Set $$X^\star:=\{x\in X:x\neq x_2~\mathrm{and}~c(x y_1)\neq c(x_1y_1)\}$$
and $$Y^\star:=\{y\in Y:y\neq y_1~\mathrm{and}~c(y x_2)\neq c(x_2y_2)\}.$$
Clearly, we have $|X^\star|,|Y^\star|\geq \delta^c(K^c_{n,n})-2\geq \frac{2}{3}n$.
We now define a directed  bipartite graph $D$ with bipartition
$( X^\star, Y^\star)$ by setting the set of arcs
$A(D):= \{
(x,y): x\in X^\star, y\in Y^\star, c(xy)\neq c(y_1x)\}\cup
\{(y,x):  x\in X^\star, y\in Y^\star,  c(xy)\neq c(yx_2)\}.$
Thus, for every $x\in X^\star$,
we have
$$d_D^+(x)\geq |\{y\in Y^\star: c(xy)\neq c(y_1x)\}|\geq|Y^\star|+\delta^c(K^c_{n,n})-1-n\geq |Y^\star|/2+1.$$
In the same way, for every $y\in Y^\star$,
we also have  $d_D^+(y)\geq|X^\star|/2+1$.
Since at most $|X^\star||Y^\star|$ arcs of $D$ are not contained in any directed 2-cycle  and
$$A(X^\star,Y^\star)+A(Y^\star,X^\star)-|X^\star||Y^\star|\geq |X^\star|+|Y^\star|,$$
at least $|X^\star|+|Y^\star|\geq \frac{4}{3}n$ arcs of $D$ are  contained in  directed 2-cycles.
From the construction of $D$,
there exist at least $\frac{4}{3}n$ edges $xy\in K^c_{n,n}[X^\star, Y^\star]$
such that $c(xy)\neq c(y_1x)~ \mathrm{and}~ c(xy)\neq c(yx_2)$.
By the definitions of $X^\star$ and $Y^\star$,
$c(x y_1)\neq c(x_1y_1)$ and $c(y x_2)\neq c(x_2y_2)$.
Hence, the path  $x_1y_1xyx_2y_2$ is PC path.
The result follows.
\end{proof}

Now we are ready to show the following lemma on the absorbing cycle.
\begin{lemma}\label{s3l4}
Let $\varepsilon>0$ and suppose that $\delta^c(G)\geq (\frac{2}{3}+\varepsilon)n$.
Then there exists an integer $n_0=n_0(\varepsilon)$ such that,
whenever $n\geq n_0(\varepsilon)$,
there exists a PC cycle $C$ of length of at most $\frac{2\varepsilon^2 n}{3}$
in $G$ such that, for each positive integer $k\leq \frac{2\varepsilon^4n}{81}$
and for every $k$  vertex-disjoint PC odd paths $Q_1, Q_2,\cdots,Q_k$ in
$G\setminus C$,
there exists a PC cycle $C'$ such that
$V(C')=V(C)\cup \bigcup_{1\leq i\leq k}V(Q_i)$.
\end{lemma}

\begin{proof}
Without loss of generality, we may assume that $\varepsilon< 2^{-3}$.
Set $\gamma=2^4\varepsilon^2/9$.
Choose $n_0(\varepsilon)$ large enough such that Lemma \ref{s3l2} holds and
such that $n_0(\varepsilon)\geq \frac{243}{4\varepsilon^3}$.
Now let $n_0=n_0(\varepsilon)$ and suppose that $n\geq n_0$.
Let $(x,y)$ be a D1-element and let $(x_1,y_1:x_2,y_2)$ be a D2-element.
From Lemma \ref{s3l1}, we have
$|\mathcal{P}(x,y)|\geq \gamma n^4$ and $|\mathcal{P}(x_1,y_1;x_2,y_2)|\geq \gamma n^4$.
From Lemma \ref{s3l2}, there exists a family $\mathcal{F}$ of vertex-disjoint PC $3$-paths
such that
\begin{equation}\label{s3equ1}
|\mathcal{F}|\leq 2^{-4}\gamma n=\frac{\varepsilon^ 2 n}{9},
\end{equation}
\begin{equation}\label{s3equ2}
|\mathcal{F}\cap \mathcal{P}(x,y)|\geq 2^{-7}\gamma^2 n=\frac{2\varepsilon^4n}{81},
\end{equation}
and
\begin{equation}\label{s3equ3}
|\mathcal{F}\cap \mathcal{P}(x_1,y_1;x_2,y_2)|\geq 2^{-7}\gamma^2 n=\frac{2\varepsilon^4n}{81}.
\end{equation}

Next we find a desired PC cycle $C$ in $G=K^c_{n,n}$.
Suppose that $|\mathcal{F}|=f$.
Let $F_1,F_2,\ldots,F_{f}$ be the PC 3-paths of $\mathcal{F}$.
For convenience, we suppose that
$$F_i=x_{2i-1}y_{2i-1}x_{2i}y_{2i},~ i\in\{1,2,\ldots,f\},$$
where $x_j\in X$ and $y_j\in Y$ for $j=1,2,\ldots, 2f$.
Moreover, we assume that $x_{2f+1}=x_1$ and $y_{2f+1}=y_1$.

We next construct a matching $\{x^{(i)}y^{(i)}:i=1,2,\ldots,f\}$
of $G\setminus \mathcal{F}$ with a desired property
by the following procedure of $l$ iterations, which generates
a sequence of matchings $M_1,M_2,\ldots, M_f$ with
$M_i=\{x^{(1)}y^{(1)}, x^{(2)}y^{(2)}, \ldots, x^{(i)}y^{(i)}\}$, $i=1,2,\ldots,f$.

{\bf Step 1:} Initially set $M_0:=\emptyset$ and $i:=1$.

{\bf Step 2:} Set $S_i:=V(\mathcal{F})\cup V(M_{i-1})\setminus \{x_{2i}, y_{2i}, x_{2i+1}, y_{2i+1}\}$
and set $G_i:=G[V(G)\setminus S_i]$.
Then pick an edge $x^{(i)}y^{(i)}\in E(G_i)$ such that
$x_{2i}y_{2i}x^{(i)}y^{(i)}x_{2i+1}y_{2i+1}$ is a PC path in $G_i$.
Set $M_i:=M_{i-1}\cup \{x^{(i)}y^{(i)}\}$.

{\bf Step 3:} If $i<f$, then set $i:=i+1$, return to Step 2.
If $i=f$, then stop the procedure.

In Step 2 of the above procedure, we have $|S_i|\leq 4f+2i\leq 6f\leq \frac{2\varepsilon^2 n}{3}$,
where the last inequality follows from (\ref{s3equ1}).
This means that $G_i$ is an edge-colored complete balanced bipartite graph with
$$\delta^c(G_i)\geq (\frac{2}{3}+\varepsilon)n- \frac{\varepsilon^2 n}{3}>
\frac{2}{3}n+2.$$
By Lemma \ref{lem3}, there exists an edge $x^{(i)}y^{(i)}\in E(G_i)$ such that
$x_{2i}y_{2i}x^{(i)}y^{(i)}x_{2i+1}y_{2i+1}$ is a PC path in $G_i$.
Consequently, the above procedure correctly generates a matching
$M_f=\{x^{(i)}y^{(i)}:i=1,2,\ldots,f\}$
of $G\setminus \mathcal{F}$.

The implementation of the above procedure also implies that
$x_{2i}y_{2i}x^{(i)}y^{(i)}x_{2i+1}y_{2i+1}$ is a PC path in $G_i\subset G$
for $i=1,2,\ldots,f$.
Thus,
$$C:=x_{1}y_{1}x_{2}y_{2}x^{(1)}y^{(1)}x_{3}y_{3}x_{4}y_{4}x^{(2)}y^{(2)}\cdots x_{2f}y_{2f}x^{(f)}y^{(f)}x_1$$
is a PC cycle in $G$ containing all the PC 3-paths of ${\cal F}$.
Note that $|C|=6f\leq \frac{2\varepsilon^2 n}{3}$, which is just what we desire.

Suppose that $\mathcal{R}$ is a family of vertex-disjoint PC odd paths in $G\setminus C$
such that $|\mathcal{R}|\leq \frac{2\varepsilon^4n}{81}$.
Recall that $\mathcal{F}=\{F_1,F_2,\ldots, F_f\}$
and $C=F_1x^{(1)}y^{(1)}F_2x^{(2)}y^{(2)}\ldots F_fx^{(f)}y^{(f)}x_1$.
We next consider a bipartite graph $G^*$ with bipartition
$(\mathcal{R},\mathcal{F})$.
For $Q\in \mathcal{R}$ and $F\in \mathcal{F}$, we define $QF$ as an edge of $G^*$
if and only if one of the following two events occurs:
(i) $Q=xy$ with $(x,y)$ being a D1-element and $F\in {\cal P}(x,y)$ is an absorbing path for $(x,y)$, and
(ii) $Q=x'y'\cdots x''y''$ with $(x',y';x'',y'')$ being a D2-element
and $F\in {\cal P}(x',y';x'',y'')$ is an
absorbing path for $(x',y';x'',y'')$.
From (\ref{s3equ2}) and (\ref{s3equ3}),
each vertex $Q\in \mathcal{R}$
has a degree of at least $\frac{2\varepsilon^4n}{81}\geq |\mathcal{R}|$.
By using Hall's theorem (see \cite{BM}, page 419),
there is a matching $M^*$ of $G^*$ that covers all the vertices of $\mathcal{R}$.

We now generate  a new cycle $C'$ of $G$ from $C=F_1x^{(1)}y^{(1)}F_2x^{(2)}y^{(2)}\ldots F_fx^{(f)}y^{(f)}x_1$ in the following way.
Let $\mathcal{R}=\{Q_1,Q_2,\ldots,Q_k\}$ and suppose that
$M^*=\{Q_iF_{i'}:i=1,2,\ldots,k\}$.
For each $i=1,2,\ldots, k$,
if $Q_i=xy$ for some D1-element $(x,y)$,
then replace the path $F_{i'}=x_{2i'-1}y_{2i'-1}x_{2i'}y_{2i'}$
in $C$ by the new path $Q'_i=x_{2i'-1}y_{2i'-1}xyx_{2i'}y_{2i'}$
containing $V(Q_i)\cup V(F_{i'})$;
and if $Q_i=x'y'\cdots x''y''$ for some D2-element
$(x',y';x'',y'')$,
then replace the path $F_{i'}=x_{2i'-1}y_{2i'-1}x_{2i'}y_{2i'}$
in $C$ by the new path $Q'_i=x_{2i'-1}y_{2i'-1}x'y'\cdots x''y''x_{2i'}y_{2i'}$
containing $V(Q_i)\cup V(F_{i'})$.

From the construction of $C'$, we may find that $C'$ is a cycle
of $G$ containing $V(C)$ and all the paths of ${\cal R}$.
Since $M^*=\{Q_iF_{i'}:i=1,2,\ldots,k\}$
is a matching of $G^*$, for each $i\in \{1,2,\ldots,k\}$,
the new path $Q'_i$ is a PC path.
Consequently, $C'$ is a PC cycle with $V(C')=V(C)\cup \bigcup_{Q\in {\mathcal{R}}}(V(Q))$.
The lemma follows.
\end{proof}

We are ready for the last step to prove Theorem \ref{t2}.
Recall that $G=K^c_{n,n}$,  $\delta^c(G)\geq (\frac{2}{3}+\varepsilon)n$,
and $(X,Y)$ is the bipartition of $G$.
Without loss of generality, we may suppose that $\varepsilon< 2^{-6}$,
$u=x_1\in X$, and $n_0=n_0(\varepsilon)\geq \frac{243}{\varepsilon^5}$.
We need to show that there exists a PC $(2l)$-cycle containing $u$
for each integer $l$ with $2\leq l\leq n$.

From Lemma \ref{h1},
there exists a PC $4$-cycle containing $u$.
For each integer $l$ with  $3\leq l\leq \frac{2}{3}\varepsilon n$,
since $\delta^c(G)\geq (\frac{2}{3}+\varepsilon)n$,
we can greedily find a PC path $P:=x_1y_1x_2y_2\cdots x_{l-1}y_{l-1}$.
Set $S:=\{x_2,y_2,x_3,y_3,\cdots, x_{l-2},y_{l-2}\}$
and $G':=G[V(G)\setminus S]$.
Note that
$$\delta^c(G')\geq (\frac{2}{3}+\varepsilon)n-\frac{2}{3}\varepsilon n \geq       (\frac{2}{3}+\frac{1}{3}\varepsilon)n\geq \frac{2}{3}|G'|+2 . $$
From Lemma \ref{lem3},
there exists an edge $xy$ of $G'$ with $x,y\in V(G')\setminus \{x_1,y_1,x_{l-1},y_{l-1}\}$
such that $y_1x_1yx y_{l-1}x_{l-1}$ is a PC path.
Hence, the cycle $x_1y_1x_2y_2\cdots x_{l-1}y_{l-1}xyx_1$ is a PC $(2l)$-cycle
containing $u$.

For each $ l\geq \frac{2}{3}\varepsilon n$,
we may suppose that $C$ is a PC $k$-cycle given by Lemma \ref{s3l4}.
Set $H:=G\setminus C$.
Then $H$ is an edge-colored complete balanced bipartite graph such that
$$\delta^c(H)\geq
(\frac{2}{3}+\varepsilon)n- \frac{\varepsilon^2 n}{3}\geq
(\frac{2}{3}+\frac{2\varepsilon}{3})n\geq \frac{2}{3}n+\frac{\varepsilon}{3}|H|.$$
Since $n_0\geq \frac{243}{\varepsilon^5}$
and $\varepsilon< 2^{-6}$, we have $\frac{\varepsilon}{3}|H|\geq \frac{\varepsilon}{3}(2n-\frac{2}{3}\varepsilon^2 n)\geq 3$.
By Corollary \ref{coro2},
there exists a family $\mathcal{F}$ of vertex-disjoint PC odd paths
with $|\mathcal{F}|\leq\lceil\frac{|H|}{\frac{\varepsilon}{3}|H|} \rceil =\lceil\frac{3}{\varepsilon}\rceil$
such that $V(H)=\bigcup_{P\in \mathcal{F}}V(P)$.
Set $\mathcal{F}=\{P_1,P_2,\ldots,P_x\}$.
If $u\in V(C)$,
then let $\mathcal{F'}=\{P'_1,P'_2,\ldots,P'_x\}$
be a family of  vertex-disjoint PC odd paths
obtained from $\mathcal{F}$ by removing some vertices in the paths of $\mathcal{F}$
such that $|\bigcup_{i\in[x]} V(P'_i)|=2l-k$.
If $u\notin V(C)$,
then we have $u\in V(\mathcal{F})$.
Since $2l\geq\frac{4}{3}\varepsilon n >\frac{2}{3}\varepsilon^2 n+2$,
we may suppose that  $\mathcal{F'}=\{P'_1,P'_2,\ldots,P'_x\}$
is  a family of  vertex-disjoint PC odd paths
obtained from $\mathcal{F}$ by removing some vertices in the paths of $\mathcal{F}$
such that $|\bigcup_{i\in[x]} V(P'_i)|=2l-k$ and $u\in \bigcup_{i\in[x]} V(P'_i)$.

Since $\lceil\frac{3}{\varepsilon}\rceil\leq\frac{2\varepsilon^4n}{81}$,
by Lemma \ref{s3l4},
there exists a PC cycle $C'$ with $V(C')=V(C)\cup\bigcup_{i\in[x]} V(P'_i)$
containing vertex $u$.
As a result, $C'$ is a PC $(2l)$-cycle containing $u$.
This completes the proof of Theorem \ref{t2}.

\section*{Acknowledgment}
This research was supported in part by the National Natural Science Foundation of China under grant numbers 11971445 and 12171440.

\end{document}